\documentclass{article}

\usepackage{amsmath}
\usepackage{amsthm}
\usepackage{amssymb}
\usepackage{enumitem}
\usepackage{amssymb}
\usepackage{amsfonts}
\usepackage{xcolor}
\usepackage[colorlinks=true,urlcolor=blue,linkcolor=blue,plainpages=false,pdfpagelabels]{hyperref}

\setenumerate{label=(\roman*)}

\newtheorem{theorem}{Theorem}[section]
\newtheorem{definition}[theorem]{Definition}
\newtheorem{lemma}[theorem]{Lemma}
\newtheorem{corollary}[theorem]{Corollary}
\newtheorem{example}[theorem]{Example}
\newtheorem{proposition}[theorem]{Proposition}

\newcommand{\abs}[1]{\left\vert#1\right\vert}
\newcommand{\ceil}[1]{\left\lceil#1\right\rceil}
\newcommand{\limn}{\lim\limits_{n \to \infty}}
\newcommand{\sumn}{\sum\limits_{n = 0}^\infty}
\newcommand{\prodn}{\prod\limits_{n = 0}^\infty}
\newcommand{\N}{\mathbb{N}}
\newcommand{\R}{\mathbb{R}}
\newcommand{\norm}[1]{\left\lVert#1\right\rVert}
\newcommand{\set}[1]{\left\{#1\right\}}
\newcommand{\Id}{\operatorname{Id}}

\newcommand{\rConvProdBeta}{\sigma}
\newcommand{\rCauchyBeta}{\chi_\beta}
\newcommand{\rCauchyLambda}{\chi_\lambda}
\newcommand{\rConvBeta}{\eta}
\newcommand{\rCauchyT}{\chi_T}

\newcommand{\cConvProdBeta}{\operatorname{C}1_q}
\newcommand{\cCauchyBeta}{\operatorname{C}2_q}
\newcommand{\cCauchyLambda}{\operatorname{C}3_q}
\newcommand{\cConvBeta}{\operatorname{C}4_q}
\newcommand{\cLimInfLambda}{\operatorname{C}5_q}
\newcommand{\cCauchyT}{\operatorname{C}6_q}

\newcommand{\notTAR}[1]{\widetilde{#1}}

\renewcommand{\phi}{\varphi}

\title{Rates of asymptotic regularity of the Tikhonov-Mann iteration for families of mappings}
\author{Hora\c tiu Cheval${}^{a}$\\[2mm]
\footnotesize ${}^a$ Research Center for Logic, Optimization and Security (LOS), Department of Computer Science, \\
\footnotesize Faculty of Mathematics and Computer Science, University of Bucharest.\\
\footnotesize Academiei 14,  010014 Bucharest, Romania\\[1mm]
\footnotesize E-mail: horatiu.cheval@unibuc.ro
}
\date{}

\begin{document}

\maketitle 

\begin{abstract}
    In this paper we generalize the strongly convergent Krasnoselskii-Mann-type 
    iteration for families of nonexpansive mappings defined recently by Bo\c{t} and Meier in Hilbert spaces
    to the abstract setting of $W$-hyperbolic spaces and we compute effective rates 
    of asymptotic regularity for our generalization. 
    This also generalizes recent results by Leu\c{s}tean and the author on the Tikhonov-Mann iteration 
    from single mappings to families of mappings.
    \\[3mm]
    \\[3mm]
\end{abstract}

\section{Introduction}

In \cite{BotMei21}, Bo\c{t} and Meier propose a strongly convergent 
Krasnoselskii-Mann-type iteration for finding a common fixed point of 
a family $(T_n : H \to H)$ of nonexpansive self-mappings of a Hilbert space.
They define the sequence $(x_n)$ by 
\begin{align}
    x_{n + 1} = (1 - \lambda_n) \beta_n x_n + \lambda_n T_n (\beta_n x_n), \label{eq:kmf}
\end{align}
where $x_0 \in H$ is an arbitrary starting point and $(\lambda_n)$, $(\beta_n)$ are sequences in $[0, 1]$.
Theorem~3.1 of \cite{BotMei21} states that, under some conditions on $(\lambda_n), (\beta_n)$ and $(T_n)$, 
\begin{align}
    \limn \norm{x_n - T_n x_n} = 0, \label{eq:kmf-T-ar}
\end{align}
and, furthermore, that if $(T_n)$ satisfies an additional 
asymptotic condition, then $(x_n)$ converges strongly to a common fixed point of $(T_n)$.
As a part of the proof of \cite[Theorem~3.1]{BotMei21}, it is also established that
\begin{align}
    \limn \norm{x_n - x_{n + 1}} = 0 \label{eq:kmf-ar}.
\end{align}

Properties \eqref{eq:kmf-ar} and \eqref{eq:kmf-T-ar} are called 
the asymptotic, respectively the $(T_n)$-asymptotic regularity of $(x_n)$, 
and are important notions in optimization and nonlinear analysis,
the first one going back to Browder and Petryshyn \cite{BroPet67} being later extended by Borwein, Reich and Shafrir \cite{BorReiSha92}.
Furthermore, they serve as key steps in many convergence proofs, including the one of \cite[Theorem~3.1]{BotMei21}.

The case when $(T_n)$ is constant in iteration \eqref{eq:kmf} was studied and proven strongly convergent by Yao, Zhou and Liou \cite{YaoZhoLio09} 
and recently by Bo\c{t}, Csetnek and Meier \cite{BotCseMei19}.
Also, the single operator case was generalized by Leu\c{s}tean and the author \cite{CheLeu22} 
to $W$-hyperbolic spaces, where quadratic rates of asymptotic regularity were obtained.
Even better, linear, rates were provided by Kohlenbach, Leu\c{s}tean and the author \cite{CheKohLeu23} in the same setting.
A further generalization in the single mapping case was introduced in \cite{DinPin21} 
under the name of alternating Halpern-Mann iteration, proven there to be strongly convergent in CAT(0) spaces,
and later also studied in $W$-hyperbolic spaces in \cite{LeuPin23}.

In this paper, we generalize iteration \eqref{eq:kmf} from Hilbert spaces 
to the much more abstract setting of $W$-hyperbolic space and prove that \eqref{eq:kmf-ar} and \eqref{eq:kmf-T-ar}
also hold for our generalization. Furthermore, our proofs are quantitative, 
providing explicit \emph{rates of asymptotic regularity} for the iteration.
Our results can also be viewed as a generalization of those in \cite{CheLeu22} 
from single mappings to families of mappings. 

The results in this paper are part of the program of \emph{proof mining} \cite{Koh08-book, Koh18} 
developed by Kohlenbach, which seeks to obtain new quantitative results via 
the proof-theoretical analysis of mathematical proofs.

\section{Preliminary notions}

\subsection{Quantitative notions}
First let us recall the quantitative notions in terms of which our main results will be expressed. 
Let $(a_n)$ be a sequence in a metric space $(X, d)$ and $a \in X$ be a point. 
A function $\phi : \N \to \N$ is a \emph{rate of convergence} for $(a_n)$ to $a$ if
\begin{align*}
    \forall k \in \N \forall n \geq \phi(k) \left(d(a_n, a) \leq \frac{1}{k + 1}\right).    
\end{align*}
A function $\chi : \N \to \N$ is a \emph{Cauchy modulus} for $(a_n)$ if 
\begin{align*}
    \forall k \in \N \forall n \geq \chi(k) \forall j \in \N \left(d(a_n, a_{n + j}) \leq \frac{1}{k + 1}\right).
\end{align*}
A \emph{rate of asymptotic regularity} for $(a_n)$ is a rate of convergence to $0$ for 
the sequence $(d(x_n, x_{n + 1}))$. 
Given a mapping $T : X \to X$, a rate of $T$-asymptotic regularity 
for $(x_n)$ is a rate of convergence to $0$ for $(d(x_n, T x_n))$,
and given a family $(T_n : X \to X)$ of self-mappings of $X$, 
a rate of $(T_n)$-asymptotic regularity of $(x_n)$ is a rate of convergence to $0$ for the sequence $(d(x_n, T_n x_n))$.

A rate of divergence for a series $\sumn b_n$ of nonnegative real numbers 
is a function $\theta : \N \to \N$ such that 
\begin{align*}
    \forall k \in \N \left(\sum_{n = 0}^{\theta(k)} b_n \geq k\right).
\end{align*}

\subsection{$W$-hyperbolic spaces}

Following \cite{CheLeu22}, we say that a \emph{$W$-space} is a metric space $(X, d)$ endowed with a mapping 
$W : X \times X \times [0, 1] \to X$. 
The intended interpretation for $W(x, y, \lambda)$ is that of an abstract convex combination of parameter $\lambda$
between the two points $x$ and $y$. Hence, instead of $W$, throughout the paper we will use 
the notation
\begin{align*}
    (1 - \lambda) x + \lambda y = W(x, y, \lambda).
\end{align*}

\begin{definition}
    A $W$-space $(X, d, W)$ is said to be a \emph{$W$-hyperbolic space} if it satisfies the following axioms,
    for all $x, y, z, w \in X$ and $\lambda, \theta \in [0, 1]$: \\[2mm]
    \begin{tabular}{lllll}
        (W1) & $d(z, (1 - \lambda) x + \lambda y) \leq (1 - \lambda) d(z, x) + \lambda d(z, y)$, \label{W1} \\[2mm] 
        (W2) & $d((1 - \lambda) x + \lambda y, (1 - \theta) x + \theta y) = \abs{\lambda - \theta} d(x, y)$, \label{W2}\\[2mm] 
        (W3) & $(1 - \lambda) x + \lambda y = \lambda y + (1 - \lambda) x$, \\[2mm] 
        (W4) & $d((1 - \lambda) x + \lambda z, (1 - \lambda) y + \lambda w) \leq (1 - \lambda) d(x, y) + \lambda d(z, w)$. \\[2mm]
    \end{tabular}  
\end{definition}
Takahashi \cite{Tak70} already studied $W$-spaces satisfying (W1), while full $W$-hyperbolic spaces were 
introduced by Kohlenbach \cite{Koh05}.
Examples of $W$-hyperbolic spaces include all normed spaces,
as well as structures from geodesic geometry, such as Busemann spaces \cite{Pap05} and CAT(0) spaces \cite{AleKapPet19, BriHae99}.

Throughout this paper, unless otherwise mentioned, $(X, d, W)$ is a $W$-hyperbolic space.
\begin{proposition}\label{prop:W-derived}
    The following hold, for all $x, y, z, w \in X$ and $\lambda, \theta \in [0, 1]$.
    \begin{enumerate}
        \item \label{prop:W-dervied:dist-to-endpoint-left} \label{prop:W-dervied:dist-to-endpoint-right} 
        $d(x, (1 - \lambda) x + \lambda y) = \lambda d(x, y)$ and $d(y, (1 - \lambda) x + \lambda y) = (1 - \lambda) d(x, y)$;
        \item \label{prop:W-derived:general}
        $d((1 - \lambda) x + \lambda z, (1 - \theta) y + \theta w) \leq (1 - \lambda) d(x, y) + \lambda d(z, w) + \abs{\lambda - \theta}d(y, w)$;
        \item \label{prop:W-derived:common-left-endpoint}
        $d((1 - \lambda) x + \lambda z, (1 - \theta) x + \theta w) \leq \lambda d(z, w) + \abs{\lambda - \theta}d(x, w)$.
    \end{enumerate}
\end{proposition}
\begin{proof}
    See \cite[Lemma~2.1]{CheLeu22}.
\end{proof}

\section{Main results}

Let $(T_n : X \to X)$ be a sequence of nonexpansive operators (i.e. $d(T_n x, T_n y) \leq d(x, y)$ for $x, y \in X$),
$(\lambda_n)$ and $(\beta_n)$ be sequences in $[0, 1]$,
and $x_0, u \in X$ be two arbitrary points. 
We define the \emph{Tikhonov-Mann iteration} associated to the family $(T_n)$
with parameters $(\lambda_n), (\beta_n)$, anchor point $u$ and starting point $x_0$ 
by
\begin{align}
    x_{n + 1} &= (1 - \lambda_n) u_n + \lambda_n T_n u_n, \quad \text{where} \label{eq:dfn-tmf-x}\\ 
    u_n &= (1 - \beta_n) u + \beta_n x_n. \label{eq:dfn-tmf-u} 
\end{align}
If $W$ is a Hilbert space, with the choice $u = 0$ we recover the iteration from \cite{BotMei21}
and if $(T_n)$ is a constant sequence we get the Tikhonov-Mann iteration from \cite{CheLeu22}. 

We will consider the following quantitative conditions on the parameters of the iteration. 

\begin{tabular}{lll}
    $(\cConvProdBeta)$ & $\prod\limits_{n = 0}^\infty \beta_{n + 1} = 0$ with rate of convergence $\rConvProdBeta$, \\[3mm] 
    $(\cCauchyBeta)$ & $\sum\limits_{n = 0}^\infty\abs{\beta_{n + 1} - \beta_n}$ is convergent with Cauchy modulus $\rCauchyBeta$, \\[3mm] 
    $(\cCauchyLambda)$ & $\sum\limits_{n = 0}^\infty\abs{\lambda_{n + 1} - \lambda_n}$ is convergent with Cauchy modulus $\rCauchyLambda$, \\[3mm] 
    $(\cConvBeta)$ & $\limn \beta_n = 1$ with rate of convergence $\rConvBeta$, \\[3mm] 
    $(\cLimInfLambda)$ & $\Lambda \in \N^*$ and $N_\Lambda \in \N$ are such that $\lambda_n \geq \frac{1}{\Lambda}$ for all $n \geq N_\Lambda$, \\[3mm]
    $(\cCauchyT)$ & $\sum\limits_{n = 0}^\infty d(T_{n + 1} u_n, T_n u_n)$ is convergent with Cauchy modulus $\rCauchyT$. \\[3mm]
\end{tabular}

These are quantitative analogues of the conditions from \cite[Theorem~2.1]{BotMei21}, 
with the caveat that the condition that $\sumn (1 - \beta_n) = \infty$ used 
in that paper is replaced with the equivalent, when $\beta_n > 0$, condition that $\prodn \beta_{n + 1} = 0$.
The reason we choose this reformulation is that it allows us to get better 
rates of asymptotic regularity,
as observed first by Kohlenbach \cite{Koh11}, who used it to obtain 
polynomial rates of asymptotic regularity for the Halpern iteration for the first time. 





For a mapping $T : X \to X$, let us denote by $\operatorname{Fix}(T) = \set{x \in X \mid Tx = x}$ 
its set of fixed points.  
In the rest of this paper, let $F = \bigcap\limits_{n \in \N} \operatorname{Fix}(T_n)$ be 
the set of common fixed points of the family $(T_n)$, and assume it to be nonempty.
The following lemma provides some useful upper bounds on the sequences involved.
\begin{lemma}\label{lem:basic-bounds}
    Let $p \in F$ be a common fixed point.
    Define 
    \begin{align}
        M = \ceil{\max\set{d(x_0, p), d(u, p)}}. \label{eq:dfn-max}       
    \end{align}
    Then, for all $n \in \N$, the following bounds hold:
    \begin{enumerate}
        \item $d(x_n, p) \leq M$ \label{lem:basic-bounds:x-p} and $d(x_n, u) \leq 2M$; \label{lem:basic-bounds:x-u}
        \item $d(u_n, p) \leq M$ \label{lem:basic-bounds:u-p} and $d(u_n, T_n u_n) \leq 2M$. \label{lem:basic-bounds:Tu-u}
    \end{enumerate}
\end{lemma}
\begin{proof} 
    All the inequalities are easily proved by adapting the proofs of \cite[Lemma~3.1]{CheLeu22},
    replacing $T$ with $T_n$.
\end{proof}

The following lemma establishes the main recursive inequality on $(x_n)$ 
that will allow us to obtain its asymptotic regularity.
\begin{lemma}\label{lem:rec-bounds-ar}
    Let $p \in F$ and $M$ be defined by \eqref{eq:dfn-max}.
    For all $n \in \N$, the following hold:
    \begin{align}
        d(u_{n + 1}, u_n) \leq&\ \ \beta_{n + 1} d(x_{n + 1}, x_n) + 2 M \abs{\beta_{n + 1} - \beta_n}; \label{lem:rec-bounds:u-succ-u} \\ 
        \begin{split}
            d(x_{n + 2}, x_{n + 1}) \leq&\ \ \beta_{n + 1}d(x_{n + 1}, x_n) + d(T_{n + 1} u_n, T_n u_n) \\ 
                &+ 2M(\abs{\lambda_{n + 1} - \lambda_n} + \abs{\beta_{n + 1} - \beta_n}) 
        \end{split} \label{lem:rec-bounds:x-succ-x}
    \end{align}
        
\end{lemma}
\begin{proof}\mbox{}
    The proofs follow those of \cite[Proposition~3.2.(6), (7)]{CheLeu22}.
    For \eqref{lem:rec-bounds:u-succ-u} we have that, for all $n \in \N$,
        \begin{align*}
            d(u_{n + 1}, u_n) 
            &\leq \beta_{n + 1} d(x_{n + 1}, x_n) + \abs{\beta_{n + 1} - \beta_n} d(u, x_n) 
            \quad\text{by Lemma \ref{prop:W-derived}.\ref{prop:W-derived:common-left-endpoint}}\\ 
            &\leq \beta_{n + 1} d(x_{n + 1}, x_n) + 2M\abs{\beta_{n + 1} - \beta_n} 
            \quad\text{from Lemma \ref{lem:basic-bounds}.\ref{lem:basic-bounds:x-u}}.
        \end{align*}

    For \eqref{lem:rec-bounds:x-succ-x}, let $n \in \N$. Then
        \begin{align*}
            d(x_{n + 2}, x_{n + 1}) 
            \leq&\ (1 - \lambda_{n + 1})d(u_{n + 1}, u_n) + \lambda_n d(T_{n + 1} u_{n + 1}, T_n u_n) \\
                &+ \abs{\lambda_{n + 1} - \lambda_n}d(u_n, T_n u_n)  
                \quad\text{by Lemma \ref{prop:W-derived}.\ref{prop:W-derived:general}} \\
            \leq&\ (1 - \lambda_{n + 1})d(u_{n + 1}, u_n) + \lambda_n d(T_{n + 1} u_{n + 1}, T_n u_n) \\ 
                &+ 2M\abs{\lambda_{n + 1} - \lambda_n}  
                \quad\text{by Lemma \ref{lem:basic-bounds}.\ref{lem:basic-bounds:Tu-u}} \\ 
            \leq&\ (1 - \lambda_{n + 1})d(u_{n + 1}, u_n) + 2M\abs{\lambda_{n + 1} - \lambda_n}  \\
                &+ \lambda_n (d(T_{n + 1} u_{n + 1}, T_{n + 1} u_n) + d(T_{n + 1} u_n, T_n u_n)) \\ 
            \leq&\ (1 - \lambda_{n + 1})d(u_{n + 1}, u_n) + 2M\abs{\lambda_{n + 1} - \lambda_n} \\ 
                & + \lambda_n (d(u_{n + 1}, u_n) + d(T_{n + 1} u_n, T_n u_n)) \\
                &\text{since $T_{n + 1}$ is nonexpansive} \\ 
            =& d(u_{n + 1}, u_n) + \lambda_n d(T_{n + 1} u_n, T_n u_n) + 2M\abs{\lambda_{n + 1} - \lambda_n} \\ 
            \leq&\ \beta_{n + 1}d(x_{n + 1}, x_n) + \lambda_n d(T_{n + 1} u_n, T_n u_n) \\ 
                &+ 2M(\abs{\lambda_{n + 1} - \lambda_n} + \abs{\beta_{n + 1} - \beta_n}) 
                \quad\text{from \eqref{lem:rec-bounds:u-succ-u}} \\ 
            \leq&\ \beta_{n + 1}d(x_{n + 1}, x_n) + d(T_{n + 1} u_n, T_n u_n) \\
                &+ 2M(\abs{\lambda_{n + 1} - \lambda_n} + \abs{\beta_{n + 1} - \beta_n}) \\ 
                &\text{because $0 \leq \lambda_n \leq 1$}.
        \end{align*}
\end{proof}

The next inequalities will be used to derive the $(T_n)$-asymptotic regularity of $(x_n)$,
and their proofs follows \cite[Proposition~3.2]{CheLeu22}.
\begin{lemma}\label{lem:rec-bounds-T-ar}
    For all $n \in \N$, the following hold:
    \begin{enumerate}
        \item $d(u_n, T_n x_n) \leq (1 - \beta_n) d(u, T_n x_n) + \beta_n d(x_n, T_n x_n)$;
        \label{lem:rec-bounds:u-Tx}
        \item $d(x_n, T_n x_n) \leq d(x_n, x_{n + 1}) + (1 - \beta_n) d(u, x_n) + (1 - \lambda_n) d(x_n, T_n x_n)$;
        \label{lem:rec-bounds:x-Tx}
        \item $\lambda_n d(x_n, T_n x_n) \leq d(x_n, x_{n + 1}) + 2M(1 - \beta_n)$.
        \label{lem:rec-bounds:x-Tx-with-lambda}
    \end{enumerate}
\end{lemma}

\begin{proof}\mbox{}
    \begin{enumerate}
        \item For all $n \in \N$,
        \begin{align*}
            d(u_n, T_n x_n) 
            \leq&\ d(u_n, (1 - \beta_n) u + \beta_n T_n x_n) + d((1 - \beta_n) u + \beta_n T_n x_n, T_n x_n) \\ 
            \leq&\ d(u_n, (1 - \beta_n) u + \beta_n T_n x_n) + (1 - \beta_n) d(u, T_n x_n) \\
            &\text{by Proposition \ref{prop:W-derived}.\ref{prop:W-dervied:dist-to-endpoint-right}} \\
            \leq&\ \beta_n d(x_n, T_n x_n) + (1 - \beta_n) d(u, T_n x_n) \quad\text{by (W4)}. 
        \end{align*}

        \item For all $n \in \N$,
        \begin{align*}
            d(x_n, T_n x_n) 
            \leq&\ d(x_n, x_{n + 1}) + d(x_{n + 1}, T_n x_n) \\
            \leq&\ d(x_n, x_{n + 1}) + (1 - \lambda_n) d(u_n, T_n x_n) + \lambda_n d(T_n u_n, T_n x_n) \\ 
                &\text{by (W1)} \\ 
            \leq&\ d(x_n, x_{n + 1}) + (1 - \lambda_n) d(u_n, T_n x_n) + \lambda_n d(u_n, x_n) \\ 
                &\text{by the nonexpansiveness of $T_n$} \\ 
            \leq&\ d(x_n, x_{n + 1}) + (1 - \lambda_n) (1 - \beta_n) d(u, T_n x_n) \\ 
                &+ (1 - \lambda_n) \beta_n d(x_n, T_n x_n) + \lambda_n d(u_n, x_n) \\ 
            =&\ d(x_n, x_{n + 1}) + (1 - \lambda_n) (1 - \beta_n) d(u, T_n x_n) \\ 
                &+ (1 - \lambda_n) \beta_n d(x_n, T_n x_n) + \lambda_n (1 - \beta_n) d(u, x_n) \\
                &\text{by (W4)} \\
            \leq&\ d(x_n, x_{n + 1}) + (1 - \lambda_n) (1 - \beta_n) (d(u, x_n) + d(x_n, T_n x_n)) \\ 
                &+ (1 - \lambda_n) \beta_n d(x_n, T_n x_n) + \lambda_n (1 - \beta_n) d(u, x_n) \\
            =&\ d(x_n, x_{n + 1}) + (1 - \beta_n) d(u, x_n) + (1 - \lambda_n) d(x_n, T_n x_n). 
        \end{align*}

        \item Starting from \ref{lem:rec-bounds:x-Tx}
        \begin{align*}
            d(x_n, T_n x_n)           &\leq d(x_n, x_{n + 1}) + (1 - \beta_n) d(u, x_n) + (1 - \lambda_n) d(x_n, T_n x_n), 
        \end{align*}
        move the last term to the left-hand side to get that
        \begin{align*}
            \lambda_n d(x_n, T_n x_n) &\leq d(x_n, x_{n + 1}) + (1 - \beta_n) d(u, x_n).             
        \end{align*}
        Apply Lemma \ref{lem:basic-bounds}.\ref{lem:basic-bounds:x-u} to get the conclusion.
    \end{enumerate}
\end{proof}

Finally, before the main theorems of the paper, 
let us give sufficient conditions for a family $(T_n)$ to satisfy $(\cCauchyT)$.

\newcommand{\cGamma}{\operatorname{C}7_q}
\newcommand{\cCauchyGamma}{\cGamma}
\newcommand{\rCauchyGamma}{{\chi_\gamma}}
\newcommand{\cLimInfGamma}{\operatorname{C}8_q}

\begin{proposition}\label{prop:cauchy-T-of-jP2-consequence}
    Let $(\gamma_n)$ be a sequence of positive reals satisfying the following conditions:
 
\begin{tabular}{lll}
    $(\cGamma)$ & $\sumn \abs{{\gamma_{n + 1} - \gamma_n}}$ is convergent with Cauchy modulus $\rCauchyGamma$; \\[3mm]
    $(\cLimInfGamma)$ & $\Gamma \in \N^*$ and $N_\Gamma \in \N$ are such that $\gamma_n \geq \frac{1}{\Gamma}$ for all $n \geq N_\Gamma$. \\[3mm]
\end{tabular}

    Suppose the family of operators $(T_n : X \to X)$ satisfies the following condition with respect to $(\gamma_n)$:
    for all $x \in X$ and $m, n \in \N$,
    \begin{align}
        d(T_m x, T_n x) \leq \frac{\abs{\gamma_{m} - \gamma_n}}{\gamma_n} d(T_n x, x). \label{eq:jP2-consequence}
    \end{align}

    Then, $(T_n)$ satisfies condition $(\cCauchyT)$ with $\chi_T$ given by
    \begin{align*}
        \rCauchyT(k) = \max\set{N_\Gamma, \rCauchyGamma(2M\Gamma(k + 1) - 1)}.
    \end{align*}
\end{proposition}

\begin{proof}
    Let $k, j \in \N$ and $n \geq \chi_T(k)$. Then,
    \begin{align*}
        \sum_{i = n + 1}^{n + j} d(T_{i + 1} u_i, T_i u_i)
        &\stackrel{\eqref{eq:jP2-consequence}}{\leq} \sum_{i = n + 1}^{n + j}\frac{\abs{{\gamma_n - \gamma_{i + 1}}}}{\gamma_i} d(u_i, T_i u_i) \\ 
        &\leq \Gamma \sum_{i = n + 1}^{n + j} {\abs{{\gamma_n - \gamma_{i + 1}}}} d(u_i, T_i u_i) \quad\text{by $(\cLimInfGamma)$}\\ 
        &\leq 2M\Gamma \sum_{i = n + 1}^{n + j} {\abs{{\gamma_n - \gamma_{i + 1}}}} \quad\text{by Lemma~\ref{lem:basic-bounds}.\ref{lem:basic-bounds:Tu-u}}\\ 
        &\leq 2M\Gamma \frac{1}{2M\Gamma(k + 1)} \quad\text{by $(\cCauchyGamma)$}\\ 
        &= \frac{1}{k + 1}
    \end{align*} 
\end{proof}

Conditions $(\cGamma)$ and $(\cLimInfGamma)$ are quantitative versions of 
those imposed in \cite[Theorem~3.1]{BotMei21} and are used there to derive \eqref{eq:jP2-consequence}.
\cite{LeuNicSip18} provides a large class of mappings satisfying Condition \eqref{eq:jP2-consequence}, named Condition $(C1)$ in that paper: 
it is shown there that if $X$ is a CAT(0) space and the family $(T_n)$ is \emph{jointly $(P_2)$ with respect to $(\gamma_n)$}, 
then $(T_n)$ satifies \eqref{eq:jP2-consequence}.

\subsection{General rates of asymptotic regularity}

\newcommand{\rCauchyTLambdaBeta}{\chi}
\newcommand{\notSum}[1]{{#1}}
\newcommand{\notProd}[1]{{{#1}}}
\newcommand{\rAR}{\Sigma}
\newcommand{\boundAwayFromZeroProd}{\psi}
\begin{theorem}\label{thm:general-ar-rates} 
    Let $p \in F$ be a common fixed point of $(T_n)$ 
    and $M$ be defined by \eqref{eq:dfn-max}.
    Furthermore, define  
    \begin{align}
        \rCauchyTLambdaBeta(k) = \max\set{\rCauchyT(2(k + 1) - 1), \rCauchyLambda(8M(k + 1) - 1), \rCauchyBeta(8M(k + 1) - 1)}. \label{eq:dfn-cauchy-T-Lambda-Beta}
    \end{align}
    Suppose conditions  $(\cConvProdBeta)$, $(\cCauchyBeta)$, $(\cCauchyLambda)$ and $(\cCauchyT)$ are satisfied
        and that $\psi_0 : \N \to \N^*$ is such that 
        \begin{align*}
            \frac{1}{\psi_0(k)} \leq \prod_{n = 0}^{\chi(3k + 2)} \beta_{n + 1}.
        \end{align*}
        Then $(x_n)$ is $(T_n)$-asymptotically regular with rate
        \begin{align*}
            \notProd{\rAR}(k) = \max\set{\rConvProdBeta(6M(k + 1)\psi_0(k) - 1), \rCauchyTLambdaBeta(3k + 2) + 1} + 1.
        \end{align*}
\end{theorem}
\begin{proof}
    We will apply Proposition~5.2.(ii) of \cite{CheLeu22},
    which is a particular case of quantitative versions of a well-known Lemma by Xu \cite{Xu02}
    proved in \cite{KohLeu12, LeuPin21},
    with
    \begin{align*}
        s_n &= d(x_n, x_{n + 1}), \\ 
        a_n &= 1 - \beta_n, \\ 
        c_n &= d(T_{n + 1} u_n, T_n u_n) + 2M(\abs{\lambda_{n + 1} - \lambda_n} + \abs{\beta_{n + 1} - \beta_n}), \\ 
        L &= 2M. 
    \end{align*}

    We proceed to show that the conditions of that proposition are fulfilled.

    \textbf{Claim:} $\rCauchyTLambdaBeta$, as defined by \eqref{eq:dfn-cauchy-T-Lambda-Beta}, is a Cauchy modulus for $\sumn c_n$.

    \textbf{Proof of claim:}
    For brevity, denote 
    $\widehat{c}_n = \sum\limits_{i = 0}^{n} c_i$, 
    $\widehat{t}_n = \sum\limits_{i = 0}^{n} d(T_{i + 1} u_i, T_i u_i)$,
    $\widehat{\lambda}_n = \sum\limits_{i = 0}^{n} \abs{\lambda_{n + 1} - \lambda_n}$ and 
    $\widehat{\beta}_n = \sum\limits_{i = 0}^{n} \abs{\beta_{n + 1} - \beta_n}$,
    so that 
    \begin{align*}
        \widehat{c}_n = \widehat{t}_n + 2M(\widehat{\lambda}_n + \widehat{\beta}_n).   
    \end{align*}
    Let $k \in \N$ and $n \geq \rCauchyTLambdaBeta(k)$ and $j \in \N$.
    Given the definition of $\rCauchyTLambdaBeta$ and the fact that 
    $\rCauchyT$, $\rCauchyLambda$ and $\rCauchyBeta$ are Cauchy moduli for 
    $(\widehat{t}_n)$, $(\widehat{\lambda}_n)$ and $(\widehat{\beta}_n)$ respectively, we get that
    \begin{align*}
        \widehat{c}_{n + j} - \widehat{c}_n 
        &= \widehat{t}_{n + j} - \widehat{t}_n + 2M(\widehat{\lambda}_{n + j} - \widehat{\lambda}_n + \widehat{\beta}_{n + j} - \widehat{\beta}_n) \\ 
        &\leq \frac{1}{2(k + 1)} + 2M\left(\frac{1}{8M(k + 1)} + \frac{1}{8M(k + 1)}\right) \\  
        &= \frac{1}{k + 1}.
    \end{align*}
    
    \hfill $\blacksquare$ 

    This claim, together with Lemma \ref{lem:rec-bounds-ar}.\eqref{lem:rec-bounds:x-succ-x} shows that we are in the conditions to apply 
    \cite[Proposition~5.2.(ii)]{CheLeu22} and obtain our result.
\end{proof}

\subsection{General rates of $(T_n)$-asymptotic regularity}

The following lemma shows that in the presence of conditions $(\cConvBeta)$, $(\cLimInfLambda)$,
the asymptotic regularity of $(x_n)$ also implies its $(T_n)$-asymptotic regularity, with an explicit translation of rates.
\begin{lemma}\label{lem:ar-to-Tn-ar}
    Suppose $\phi$ is a rate of asymptotic regularity for $(x_n)$
    and assume that conditions $(\cConvBeta)$, $(\cLimInfLambda)$ hold. 
    Then, $\notTAR{\phi}$ defined by  
    \begin{align*}
        \notTAR{\phi}(k) = \max\set{N_\Lambda, \phi(2\Lambda(k + 1) - 1), \rConvBeta(4M\Lambda(k + 1) - 1)}
    \end{align*}
    is a rate of $(T_n)$-asymptotic regularity for $(x_n)$.
\end{lemma}
\begin{proof}
    Let $k \in \N$ and $n \geq \notTAR{\phi}(k)$. We need to show that $d(x_n, T_n x_n) \leq \frac{1}{k + 1}$.

    As $n \geq N_\Lambda$, Condition $(\cLimInfLambda)$ shows that $\lambda_n > 0$ and that
    \begin{align}
        \frac{1}{\lambda_n} \leq \Lambda. \label{eq:ar-to-Tn-ar:prf-lambda}
    \end{align}

    Since $n \geq \phi(2\Lambda(k + 1) - 1)$, the fact that $\phi$ is a rate of asymptotic regularity yields
    \begin{align}
        d(x_n, x_{n + 1}) \leq \frac{1}{2\Lambda(k + 1)}. \label{eq:ar-to-Tn-ar:prf-lambda-xn}
    \end{align}

    Finally, because $n \geq \rConvBeta(4M\Lambda(k + 1) - 1)$ and $\rConvBeta$ is a rate of convergence 
    for $\limn (1 - \beta_n) = 0$, we get
    \begin{align}
        1 - \beta_n \leq \frac{1}{4M\Lambda(k + 1)}. \label{eq:ar-to-Tn-ar:prf-lambda-beta}
    \end{align}

    From Lemma \ref{lem:rec-bounds-ar}.\ref{lem:rec-bounds:x-Tx-with-lambda}, we know that 
    \begin{align*}
        d(x_n, T_n x_n) \leq \frac{1}{\lambda_n} d(x_n, x_{n + 1}) + \frac{1}{\lambda_n} 2M(1 - \beta_n),
    \end{align*}
    which, together with \eqref{eq:ar-to-Tn-ar:prf-lambda}, \eqref{eq:ar-to-Tn-ar:prf-lambda-xn} and \eqref{eq:ar-to-Tn-ar:prf-lambda-beta} yields
    \begin{align*}
        d(x_n, T_n x_n) \leq \Lambda\frac{1}{2\Lambda(k + 1)}  + \Lambda 2M\frac{1}{4M\Lambda(k + 1)} = \frac{1}{k + 1}
    \end{align*}
    thus proving the claim.
\end{proof}

\begin{theorem}\label{thm:general-Tn-ar-rates}
    Suppose conditions $(\cConvProdBeta)$, $(\cCauchyBeta)$, $(\cCauchyLambda)$, $(\cConvBeta)$, $(\cLimInfLambda)$ and $(\cCauchyT)$ hold. 
    Let $\notProd{\rAR}$ be defined as in Theorem \ref{thm:general-ar-rates}.
    Then $(x_n)$ is $(T_n)$-asymptotically regular with rate 
    \begin{align}
        \notTAR{\notProd{\rAR}}(k) = \max\set{N_\Lambda, \notProd{\rAR}(2\Lambda(k + 1) - 1), \rConvBeta(4M\Lambda(k + 1) - 1)}.
    \end{align}



\end{theorem}
\begin{proof}
    Apply Lemma \ref{lem:ar-to-Tn-ar} together with 
    Theorem \ref{thm:general-ar-rates}.
\end{proof}
Other than the parameters $(\lambda_n)$, $(\beta_n)$ of the iteration, 
the rates of ($(T_n)$-)asymptotic regularity we obtained depend weakly on the space $X$ and on the mappings $(T_n)$,
only via $M$ and $\rCauchyT$. The rates are less uniform compared to the rates obtained in the single mapping 
case in \cite{CheLeu22}, where only $M$ is present. Indeed, if $(T_n)$ is constant, 
then $\rCauchyT$ can simply be taken as $k \mapsto 0$. 
However, Proposition \ref{prop:cauchy-T-of-jP2-consequence} shows that for a large class of mappings, 
the dependence on $\rCauchyT$ can be reduced to one only on the real parameters $(\gamma_n)$.

In the following example we compute explicit rates for a concrete choice of $(\beta_n)$, $(\lambda_n)$ and $(\gamma_n)$.
The rates obtained for this example are the same as those obtained for the single mapping case in \cite[Corollary~4.3]{CheLeu22},
which demonstrates how under the hypotheses of Proposition~\ref{prop:cauchy-T-of-jP2-consequence},
the additional dependence on the familiy $(T_n)$ is eliminated.

\begin{example}\label{cor:example}
    Let $\lambda_n = \lambda \in (0, 1)$, $\beta_n = 1 - \frac{1}{n + 1}$ and $\gamma_n = 1 + \frac{1}{n + 1}$
    and consider the iteration $(x_n)$ given by \eqref{eq:dfn-tmf-x} with these parameters
    and let $(T_n)$ be a family of nonexpansive mappings satisfying 
    $\eqref{eq:jP2-consequence}$ with respect to $(\gamma_n)$.
    Then 
    \begin{enumerate}
        \item \label{cor:example-ar}
        $(x_n)$ is asymptotically regular with rate 
        \begin{align*}
            k \mapsto 144M^2(k + 1)^2 - 6M(k + 1);
        \end{align*}
        \item \label{cor:example-T-ar}
        $(x_n)$ is $(T_n)$-asymptotically regular with rate 
        \begin{align*}
            k \mapsto 576M^2\ceil{\frac{1}{\lambda}}^2(k + 1)^2 - 12M\ceil{\frac{1}{\lambda}}(k + 1).
        \end{align*}
    \end{enumerate}
\end{example}
\begin{proof}\mbox{}
    \begin{enumerate}
    \item 
        We will apply Theorem \ref{thm:general-ar-rates}. 
        Let us first show that its assumptions are fulfilled.
        Because
        \begin{align*}
            \prod_{i = 0}^n \beta_{n + 1} = \frac{1}{n + 2}, \quad \sum_{i = 0}^n \abs{\beta_{i + 1} - \beta_i} = 1 - \frac{1}{n + 2} \ \text{and} \ \sum_{i = 0}^n \abs{\lambda_{i + 1} - \lambda_i} = 0,
        \end{align*}
        we get that $(\cConvProdBeta)$, $(\cCauchyBeta)$ and $(\cCauchyLambda)$ are satified, respectively, with 
        \begin{align*}
            \rConvProdBeta(k) = k, \quad \rCauchyBeta(k) = k, \quad \rCauchyLambda(k) = 0.
        \end{align*}
        $(\cConvBeta)$ holds with 
        \begin{align*}
            \rConvBeta(k) = k.
        \end{align*}
        $(\cLimInfLambda)$ is satified with,
        \begin{align*}
            N_\Lambda = 0, \quad \Lambda = \ceil{\frac{1}{\lambda}}.
        \end{align*}
        Furthermore, as 
        \begin{align*}
            \sum_{i = 0}^n \abs{\gamma_{i + 1} - \gamma_i} = 1 - \frac{1}{n + 2},
        \end{align*}
        $(\gamma_n)$ fulfills $(\cCauchyGamma)$ and $(\cLimInfGamma)$ with 
        \begin{align*}
            \Gamma = 1, \quad N_\Gamma = 0 \quad \text{and} \quad \rCauchyGamma(k) = k.
        \end{align*}
        Thus, by Proposition \ref{prop:cauchy-T-of-jP2-consequence}, it follows that $(T_n)$ satisfies $(\cCauchyT)$ with 
        \begin{align*}
            \rCauchyT(k) = \max\set{N_\Gamma, 2M\Gamma(k + 1) - 1} = 2M(k + 1) - 1
        \end{align*}
        It follows that $\chi$ defined by \eqref{eq:dfn-cauchy-T-Lambda-Beta} is equal in this case to 
        \begin{align*}
            \chi(k) &= \max\set{\rCauchyT(2(k + 1) - 1), \rCauchyLambda(8M(k + 1) - 1), \rCauchyBeta(8M(k + 1) - 1)} \\ 
                    &= \max\set{4M((k + 1)) - 1, 0, 8M(k + 1) - 1} \\ 
                    &= 8M(k + 1) - 1
        \end{align*}
        Finally, $\psi_0$ can be taken as $\psi_0(k) = \chi(3k + 2) = 24M(k + 1) - 1$.
        Thus,
        \begin{align*}
            \notProd{\Sigma}(k) &= \max\set{\rConvProdBeta(6M(k + 1)\psi_0(k) - 1), \chi(3k + 2) + 1} + 1 \\ 
            &= 144M^2(k + 1)^2 - 6M(k + 1)
        \end{align*}
    \item 
        Apply Theorem \ref{thm:general-Tn-ar-rates} with the rate $\notProd{\Sigma}$ from \ref{cor:example-ar}, we get that 
        \begin{align*}
            \notTAR{\notProd{\Sigma}}(k) &= \max\set{N_\Lambda, \notProd{\Sigma}(2\Lambda(k + 1) - 1), \rConvBeta(4M\Lambda(k + 1) - 1)} \\ 
            &= 576M^2\ceil{\frac{1}{\lambda}}^2(k + 1)^2 - 12M\ceil{\frac{1}{\lambda}}(k + 1) \\ 
        \end{align*}
        
    \end{enumerate}

\end{proof}

\subsection{Linear rates of ($(T_n)$-)asymptotic regularity}

In this section, we compute linear rates of asymptotic regularity for iteration \eqref{eq:kmf}, 
by applying a lemma on real numbers introduced by Sabach and Shtern \cite{SabSht17},
that was originally used there to obtain the linear asymptotic regularity of a viscosity-type Halpern iteration,
and was recently employed to the same end for a variety of iterations \cite{CheKohLeu23, LeuPin23, CheLeu23}.

The following is a particular case of \cite[Lemma~3]{SabSht17}, as reformulated in \cite{LeuPin23}.
\begin{lemma}\label{lem:Sabach-Shtern}
    Let $L > 0$, 
    and define, for all $n \in \N$, $a_n = \frac{2}{n + 2}$. 
    Suppose $(s_n)$ is a sequence of nonnegative reals such that $s_0 \leq L$ and
    that 
    \begin{align}
        s_{n + 1} \leq (1 - a_{n + 1}) s_n + (a_n - a_{n + 1}) L \label{eq:Sabach-Shtern}
    \end{align}
    for all $n \in \N$. Then, 
    \begin{align*}
        s_n \leq \frac{2 L}{n + 2}
    \end{align*}
    for all $n \in \N$.
\end{lemma}

\begin{theorem}
    Let $\beta_n = 1 - \frac{2}{n + 2}$, $\lambda_n = \lambda \in (0, 1)$ and $\gamma_n = \frac{n + 3}{n + 2}$
    and suppose the family $(T_n)$ satisfies \eqref{eq:jP2-consequence} with respect to $(\gamma_n)$.
    Then, the iteration \eqref{eq:kmf} with parameters $(\beta_n)$ and $(\lambda_n)$ satisfies, for all $n, m \in \N$,
    \begin{align*}
        d(x_n, x_{n + 1}) &\leq \frac{6M}{n + 2}; \\
        d(x_n, T_n x_n) &\leq \frac{10M}{\lambda(n + 2)}; \\
        d(x_n, T_m x_n) &\leq \frac{20M}{\lambda(n + 2)}.
    \end{align*}
    Thus, the mappings $k \mapsto 6M(k + 1) - 2$ , 
    $k \mapsto 10M\ceil{\frac{1}{\lambda}}(k + 1) - 2$ and 
    $k \mapsto 20M\ceil{\frac{1}{\lambda}}(k + 1) - 2$
    are rates of 
    asymptotic, $(T_n)$-asymptotic and $T_m$-asymptotic regularity for $(x_n)$, respectively.
\end{theorem}
\begin{proof}
    For the first inequality, we will apply Lemma \ref{lem:Sabach-Shtern} with 
    \begin{align*}
        s_n &= d(x_n, x_{n + 1}), \\ 
        a_n &= 1 - \beta_n, \\ 
        L &= 3M.
    \end{align*}
    Note first that 
    \begin{align*}
        \abs{\gamma_{n + 1} - \gamma_n} &= \frac{1}{(n + 2)(n + 3)}, \\
        \frac{\abs{\gamma_{n + 1} - \gamma_n}}{\gamma_n} &= \frac{n + 2}{n + 3} \cdot \frac{1}{(n + 2)(n + 3)} = \frac{1}{(n + 3)^2}
        \intertext{and that} 
        \beta_{n + 1} - \beta_n &= \frac{2}{(n + 2)(n + 3)}.
    \end{align*}
    We now show that the main condition \eqref{eq:Sabach-Shtern} is satisfied.
    By Lemma \ref{lem:rec-bounds-ar} and \eqref{eq:jP2-consequence}, for all $n \in \N$,
    \begin{align*} 
        s_{n + 1} &\leq \beta_{n + 1} s_n + 2M(\beta_{n + 1} - \beta_n) + d(T_{n + 1} u_n, T_n u_n) \\ 
        &\leq \beta_{n + 1} s_n + 2M(\beta_{n + 1} - \beta_n) + \frac{\abs{\gamma_{n + 1} - \gamma_n}}{\gamma_n} d(T_n u_n, u_n) \\
        &\leq \beta_{n + 1} s_n + 2M(\beta_{n + 1} - \beta_n) + 2M \frac{\abs{\gamma_{n + 1} - \gamma_n}}{\gamma_n} \\
        &= \beta_{n + 1} s_n + (\beta_{n + 1} - \beta_n) \left(2M + 2M \frac{\abs{\gamma_{n + 1} - \gamma_n}}{\gamma_n (\beta_{n + 1} - \beta_n)}\right) \\
        &= \beta_{n + 1} s_n + (\beta_{n + 1} - \beta_n) \left( 2M + \frac{2M}{(n + 3)^2} \cdot \frac{(n + 2)(n + 3)}{2} \right) \\
        &= \beta_{n + 1} s_n + (\beta_{n + 1} - \beta_n) \left( 2M + M\frac{n + 2}{n + 3} \right) \\
        &\leq \beta_{n + 1} s_n + (\beta_{n + 1} - \beta_n) L \\
    \end{align*}
    thus proving the first claim. 
    For the second claim,
    using the previous and Lemma \ref{lem:rec-bounds-T-ar}.\ref{lem:rec-bounds:x-Tx-with-lambda}, we have
    \begin{align*}
        d(x_n, T_n x_n) &\leq \frac{1}{\lambda} d(x_n, x_{n + 1}) + \frac{4M}{\lambda(n + 2)} \\ 
        &\leq \frac{6M}{\lambda(n + 2)} + \frac{4M}{\lambda(n + 2)} = \frac{10M}{\lambda(n + 2)}.
    \end{align*}
    Finally, 
    \begin{align*}
        d(x_n, T_m x_n) 
        &\leq d(x_n, T_n x_n) + \frac{\abs{\gamma_n - \gamma_m}}{\gamma_m} d(x_n, T_n x_n) \\ 
        &= d(x_n, T_n x_n) + \frac{\abs{n - m}}{(m + 2)(n + 3)} d(x_n, T_n x_n) \\ 
        &\leq 2 d(x_n, T_n x_n) 
    \end{align*}
    and the claim follows from the previous result.

\end{proof}

    




\subsection{Relation to the modified Halpern iteration}

In this section we show that the relation between 
the modified Halpern, originally introduced in \cite{KimXu05}, and the Tikhonov-Mann iterations studied in \cite{CheKohLeu23}
for the single mapping case extends to families of mappings: the ($(T_n)$-)asymptotic regularity of one iteration 
implies that of the other, with an explicit translation of rates.

The modified Halpern iteration as defined for $W$-hyperbolic spaces in \cite{CheKohLeu23}
can naturally be extended to families of mappings as follows. 
Define, for all $n \in \N$,
\begin{align*}
    y_{n + 1} &= (1 - \beta_{n + 1}) u + \beta_{n + 1} v_n, \quad\text{where} \\ 
    v_n &= (1 - \lambda_n) y_n + \lambda_n T_n y_n.
\end{align*}

\begin{proposition}
    Assume that $y_0 = (1 - \beta_0) u + \beta_0 x_0$. 
    Then, for all $n \in \N$,
    \begin{align*}
        u_n = y_n \quad \text{and} \quad x_{n + 1} = v_n.
    \end{align*}
\end{proposition}
\begin{proof}
    The proof goes by induction on $n$ just like the proof of \cite[Proposition~3.2]{CheKohLeu23},
    replacing $T$ with $T_n$.
\end{proof}



\begin{proposition}
    Assume that $(\cConvBeta)$ holds,
    that $y_0 = (1 - \beta_0) u + \beta_0 x_0$,
    define 
    \begin{align*}
        \alpha(k) = \eta(2M(k + 1) - 1),
    \end{align*}
    and suppose $\Sigma$ is a rate of ($(T_n)$)-asymptotic 
    regularity for one of the sequences $(x_n)$ or $(y_n)$.
    Then, the other sequence is ($(T_n)$-) asymptotically regular 
    with rate 
    \begin{align*}
        \Sigma'(k) = \max\set{\alpha(3k + 2), \Sigma(3k + 2)}
    \end{align*}
\end{proposition}
\begin{proof}
    Noting that 
    \begin{align*}
        d(x_n, u_n) \stackrel{\text{Proposition~\ref{prop:W-derived}.\ref{prop:W-derived:common-left-endpoint}}}{=} (1 - \beta_n) d(x_n, u) 
        \stackrel{\text{Lemma~\ref{lem:basic-bounds}.\ref{lem:basic-bounds:x-u}}}{\leq} (1 - \beta_n) 2M,  
    \end{align*}
    the fact that $\limn d(x_n, u_n) = 0$ with rate of convergence $\alpha$ is proved as in \cite[Lemma~4.1]{CheKohLeu23}.
    Using this, the fact that $\Sigma'$ is a rate of ($(T_n)$-)asymptotic regularity is 
    proved the same as in \cite[Proposition~4.2]{CheKohLeu23}, replacing $T$ with $T_n$.
\end{proof}


\newcommand{\res}[2]{\operatorname{J}_{{#1}{#2}}}
\newcommand{\gr}[1]{\operatorname{gra}(#1)}
\newcommand{\inner}[2]{\left\langle#1, #2\right\rangle}
\newcommand{\zer}[1]{\operatorname{zer}(#1)}

\section{Rates of asymptotic regularity for the Tikhonov-forward-backward algorithm with variable step-size}

Let us first recall some notions from convex optimization and monotone operator theory.
Let $H$ be a Hilbert space. For a set-valued operator $A : H \rightrightarrows H$,
its graph is $\gr{A} = \set{(x, u) \in H \mid u \in A x}$ and 
its set of zeroes is $\zer{A} = \set{x \in H \mid 0 \in A x}$.
Given $A : H \rightrightarrows H$, the inverse of $A$ is the operator $A^{-1} : H \rightrightarrows H$, $A^{-1} u = \set{x \mid u \in A x}$
and the scaling of $A$ by $\gamma \in \R$ is the operator $\gamma A : H \rightrightarrows H$, $(\gamma A) x = \set{\gamma u \mid u \in A x}$.
Given two set-valued operators $A, B : H \rightrightarrows H$,
their sum, $A + B : H \rightrightarrows H$, is defined by $(A + B)x = \set{u + v \mid u \in A x, v \in B x}$
and their composition, $A B : H \rightrightarrows H$, by $(A B) x = \set{v \in H \mid \text{ there exists } u \in A x \text{ such that } v \in B u}$.

We say that $A : H \rightrightarrows H$ is \emph{monotone} if, for all $(x, u), (y, v) \in \gr{A}$, 
we have that $\inner{x - y}{u - v} \geq 0$. $A$ is called \emph{maximally monotone} if 
there exists no other monotone operator $B : H \rightrightarrows H$ 
such that $\gr{A} \subsetneq \gr{B}$. 
For a maximally monotone operator $A$, its \emph{resolvent} of order $\gamma > 0$ is defined by   
\begin{align*}
    \res{\gamma}{A} = (\Id + \gamma A)^{-1}.
\end{align*}
and it is known to be a single-valued, nonexpansive mapping,
where $\Id : H \to H$ is the identity mapping.
A single-valued operator $B : H \to H$ is said to be $\beta$-cocoercive for some $\beta > 0$ if,
for all $x, y \in H$, $\inner{x - y}{B x - B y} \geq \beta\norm{T x - T y}^2$.

The \emph{forward-backward} algorithm is one of the procedures widely employed for finding 
a point in $\zer{A + B}$ where $A$ is maximally monotone and $B$ is cocoercive. 
We refer to \cite{BauCom10} for a more detailed account of this algorithm and the theory of monotone operators in general.

Based on their strongly convergent Krasnoselskii-Mann iteration for families of mappings, 
the authors of \cite{BotMei21} define a version of the forward-backward algorithm 
with variable step-size and prove its strong convergence. 
In this section, we give a generalized verion of this iteration based on \eqref{eq:dfn-tmf-x} 
and compute rates of ($(T_n)$-)asymptotic regularity for it. 

Let in the following $A : H \rightrightarrows H$ be maximally monotone and $B : H \to H$ be $\beta$-cocoercive for some $\beta > 0$. 
The Tikhonov-forward-backward algorithm with variable step size associated with $A$ and $B$ is defined by
\begin{align*}
    x_{n + 1} &= (1 - \lambda_n) u_n + \lambda_n \res{\gamma_n}{A} (u_n - \gamma_n B u_n), \quad\text{where} \\ 
    u_n &= (1 - \beta_n) u + \beta_n x_n, 
\end{align*} 
where $(\beta_n) \subset (0, 1]$, $(\gamma_n) \subset (0, 2\beta)$ and 
$\lambda_n \subset (0, \frac{1}{\alpha_n}]$.

\begin{corollary}
    Assume $(\cConvProdBeta)$, $(\cCauchyBeta)$, $(\cCauchyLambda)$, $(\cCauchyGamma)$, $(\cLimInfGamma)$ are satisfied. 
    Let $\rCauchyT : \N \to \N$ be defined by:
    \begin{align}
        \rCauchyT(k) = \max\set{N_\Gamma, \rCauchyGamma(2M\Gamma(k + 1) - 1)}. \label{eq:cor-fb-cauchy-T}
    \end{align}
    Then $(x_n)$ is asymptotically regular with rate $\notProd{\Sigma}$ 
    as defined as in Theorem~\ref{thm:general-ar-rates}.

        

\end{corollary} 
\begin{proof}
    For any $n \in \N$, define $T_n = \res{\gamma_n}{A} (\Id - \gamma_n B)$,
    so that $x_n$ can be written as 
    \begin{align*}
        x_{n + 1} = (1 - \lambda_n) u_n + \lambda_n T_n u_n. 
    \end{align*}
    It is known \cite[Proposition~26.1.(v)]{BauCom10} that $T_n$ is $\frac{2\beta}{4\beta - \gamma_n}$-averaged,
    and hence nonexpansive. 
    Lemma~3.2 of \cite{BotMei21} then shows that $(T_n)$ satisfies \eqref{eq:jP2-consequence} with respect to $(\gamma_n)$,
    and thus, by Proposition~\ref{prop:cauchy-T-of-jP2-consequence}, $(\cCauchyT)$ holds with 
    $\rCauchyT$.
    All the hypotheses to apply Theorem \ref{thm:general-ar-rates} are thus satisfied, and hence the result follows.
\end{proof}

\section{Acknowledgements}

The author is grateful to Lauren\c{t}iu Leu\c{s}tean for initially suggesting the investigation of the topic herein
and for the later comments that improved the paper. 


\bibliographystyle{plainurl}

\end{document}